\DeclareSymbolFont{AMSb}{U}{msb}{m}{n}
\documentclass[12pt,oneside,noamsfonts]{amsart}
\usepackage[colorlinks=true,citecolor=blue]{hyperref}
\usepackage[scale=0.7]{geometry}
\usepackage[utf8]{inputenc}
\usepackage{graphicx}
\usepackage{amsaddr}

\newtheorem{theorem}{Theorem}[section]
\newtheorem*{theorem*}{Theorem}
\newtheorem{lemma}[theorem]{Lemma}

\theoremstyle{definition}

\theoremstyle{remark}

\DeclareMathOperator{\conv}{conv}
\DeclareMathOperator{\bd}{bd}
\newcommand{\abs}[1]{\lvert #1 \rvert}

\begin{document} 

\title{Meissner Polyhedra}
\author{Luis Montejano}
\address[L. Montejano]{Instituto de Matemáticas, Unidad Juriquilla, UNAM}
\email{luis@matem.unam.mx}
\author{Edgardo Roldán-Pensado}
\address[E. Roldán-Pensado]{Instituto de Matemáticas, Unidad Juriquilla, UNAM}
\email{e.roldan@im.unam.mx}


\begin{abstract}
In this paper we develop a concrete way to construct bodies of constant width in dimension three. They are constructed from special embeddings of self-dual graphs.
\end{abstract}

\maketitle

\section{Introduction and preliminaries}\label{sec:intro}

Constant width bodies and their properties have been known for centuries. L. Euler, in fact, studied them under the name ``orbiforms''. He was interested in constant width curves whose boundaries could be represented as the evolute of a hypocycloid. Nearly a hundred years later, in 1875, Franz Reuleaux \cite{R} published a book on kinematics, in which he mentioned constant width curves and gave some examples. He later gave the construction of what might be considered the simplest constant width curve which is not a circle, and which today bears his name. Although we know of many procedures to construct curves with constant width, the same is not true for their higher dimensional analogues.

By a theorem of Pál, we know that every subset of $R^n$ of diameter $1$ is contained in a body of constant width \cite{CG}.
Sallee \cite{S}, and Lachand and Outdet \cite{LO}, among others, gave non-constructive procedures to find them, but besides the two Meissner solids \cite{Me1}, and the obvious constant width bodies of revolution, there is no concrete example in the literature of a constant width body of dimension greater than $2$ or a concrete finite procedure to construct one.

The purpose of this paper is to construct concrete examples of constant width bodies in dimension three. They are constructed from some special embeddings of self-dual graphs. We also give a finite procedure to construct a $3$-dimensional constant width body from a Reuleaux polygon in dimension $2$.

For a treatment of constant width bodies and their properties, see the books \cite{BF,BY} and the surveys \cite{CG,HM}.

\section{Ball polyhedra}\label{sec:ballpoly}

The main goal of this section is to study the geometry of the intersection of finitely many congruent balls. Ball polyhedra are natural objects of study for several important problems of discrete geometry such as the Grünbaum-Heppes-Straszewicz theorem on the maximal number of diameters of finite point sets in $R^3$ \cite{Gr,He,Str}, the Kneser-Poulsen conjecture \cite{BC}, the proof of the Borsuk conjecture for finite point sets \cite{AP}, and the analogue of Cauchy's rigidity theorem for triangulated ball polyhedra \cite{BN}.
For a good reference about ball polyhedra, see \cite{KMP,BLNP,BN}.

Given $x\in R^n$ and $h>0$ we write $B(x,h)$ for the closed ball of radius $h$ centered at $x$ and $S(x,h)$ for the sphere of radius $h$ centered at $x$.

A \emph{ball polyhedron} in $R^n$ is the intersection of finitely many but at least $n$ solid spheres of radius $h>0$.
Let us consider a ball polyhedron $\Phi=\bigcap_{x\in X} B(x,h)$. Assume first that $\Phi$ is $3$-dimensional, $\Phi$ has non empty interior and for any proper subset $X^\prime\subset X$, $\Phi\neq\bigcap_{x\in X^\prime} B(x,h)$. We wish to describe the boundary of $\Phi$.

The points of $\bd\Phi$ are either singular or regular. The singular points can be divided into two sets: the $0$-singular and the $1$-singular.
The $0$-singular points are the the points $y\in\bd\Phi$ such that $X\subset B(y,h)$ and there are points $x_1,x_2,x_3\in X$ with $x_1-y,x_2-y,x_3-y$ linearly independent.
The set of $1$-singular point of $\bd\Phi$ is given by the set of $y\in R^3$ such that $X\subset B(y,h)$ and $S(y,h)\cap X$ contains at least two points and is contained in some great circle of $S(y,h)$. Lastly, the set of regular points of the boundary of $\Phi$ is given by the points $y\in \bd\Phi$ such that $X\subset B(y,h)$ and $\abs{S(y,h)\cap X}=1$.

Consequently, the set of singular points of the boundary of $\Phi$, $S(\Phi)$, consists of an embedding of a graph $G_\Phi$ whose vertices $V(\Phi)$ are the $0$-singular points of $\bd\Phi$ and whose edges correspond to sub-arcs of circles, each one of them joining a pair of points in $V(\Phi)$. Given adjacent vertices $a,b\in V(\Phi)$, we denote by $\widetilde{ab}$ the sub-arc of circle joining $a$ with $b$.
The complement of $S(\Phi)$ in the boundary of $\Phi$ consists of the regular points of $\Phi$, whose components are spherically convex open subsets of a sphere of radius $h$. In this way, a \emph{face} of $\Phi$ is defined as
$$S(x,h)\cap \Phi,$$
where $x\in X$.

It is not difficult to see that the closure of every component of the set of regular points, that is, the closure of every component of $\Phi-S(\Phi)$, is a face of $\Phi$. Furthermore, by Proposition 4.2 of \cite{KMP}, each face $\sigma \subset S(x,h)$ of $\Phi$ is spherically convex in $S(x,h)$. A $3$-dimensional ball polyhedron $\Phi$ is \emph{standard} if the intersection of two faces is either empty, a vertex of $G_\Phi$ or a single edge of $G_\Phi$. In fact, the following is proved in Section 6 of \cite{KMP}.

\begin{theorem} 
The graph $G_\Phi$ of a standard $3$-dimensional ball polyhedron $\Phi$ is simple, planar and $3$-connected.
\end{theorem}

As a consequence, we have the Euler-Poincaré formula $v-e+f=2$ for any $3$-dimensional ball polyhedron with $v$ vertices, $e$ edges and $f$ faces. 

Now let $\Phi\subset R^n$ be a ball polyhedron. Suppose $\Phi=\bigcap_{x\in X} B(x,h)$ but for any proper subset $X^\prime\subset X$, $\Phi\not= \bigcap_{x\in X^\prime} B(x,h)$. A \emph{supporting sphere} $S^l$ is a sphere of dimension $l$, where $0\leq l\leq n-1$, which can be obtained as the intersection of some of the spheres in $\{S(x,h)\}_{x\in X}$.
We say that an $n$-dimensional ball polyhedron $\Phi$ is \emph{standard} if for any supporting sphere $S^l$ the intersection $\Phi\cap S^l$ is spherically convex in $S^l$. If this is so, we call $\sigma$ a \emph{face} of $\Phi$ if $\sigma= \Phi\cap S^l$, for some supporting sphere $S^l$ of $\Phi$, where the dimension of $\sigma$ is $l$. Bezdek et. al. \cite{BLNP} proved the following theorem.
 
\begin{theorem} 
Let $\Phi\subset R^n$ be a standard ball polyhedron. Then the faces of $\Phi$ form the closed cells of a finite CW-decomposition of the boundary of $\Phi$. Furthermore, we have following Euler-Poincaré formula
$$1+ (-1)^{n+1}= \sum_{i=0}^n (-1)^i f_i(\Phi)$$
where $f_i(\Phi)$ denotes the number or $i$-dimensional faces of $\Phi$.
\end{theorem}

\section{Reuleaux polyhedra}\label{sec:reupoly}

Following the ideas of Sallee \cite{S}, we define a \emph{Reuleaux Polyhedron} as a convex body $\Phi\subset R^n$ satisfying the following properties:
\begin{itemize}
\item there is a set $X\subset R^n$ with $\Phi=\bigcap_{x\in X} B(x,h)$,
\item $\Phi$ is a standard ball polyhedron, and
\item the set $V(\Phi)$ of $0$-singular points of $\bd\Phi$ is $X$.
\end{itemize}

In dimension $2$, Reuleaux polyhedra are exactly the Reuleaux polygons \cite{BY} and it is well known that Reuleaux polyhedra, except in dimension $2$, are not bodies of constant width. The simplest example of Reuleaux polyhedron is the Reuleaux tetrahedron which is the $3$-dimensional analogue of the Reuleaux triangle, that is, the intersection of $4$ solid spheres of radius $h$ centered at the vertices of a regular tetrahedron of side length $h$. The corresponding self-dual graph for the Reuleaux tetrahedron is the complete graph $K_4$ with $4$ vertices. 
$3$-dimensional Reuleaux polyhedra will be the key to construct, in Section \ref{sec:meipoly}, examples of $3$-dimensional constant width bodies.

\begin{theorem}\label{thm:reuleaux}
Let $\Phi\subset R^3$ be a Reuleaux polyhedron. Then, $G_\Phi$ is a self-dual graph, where the automorphism $\tau$ is given by: $\tau(x)=S(x,h)\cap \Phi,$ for every $x\in X$. Furthermore, $\tau$ is an involution, that is, a vertex $x$ belongs to the cell $\tau (y)$ if and only if the vertex $y$ belongs to the cell $\tau (x)$.
\end{theorem}

\begin{proof} Let $\Phi=\bigcap_{x\in X} B(x,h)$ and suppose that $X$ is the set of $0$-singular points of $\Phi$. Take $x,y \in X$, the corresponding dual faces are $\tau(x) = S(x,h)\cap\Phi$ and $\tau(y)=S(y,h)\cap\Phi$. 
Assume first that the faces $\tau(x)$ and $\tau(y)$ intersect on the edge $\widetilde{ab}$ of $G_\Phi$, where $a,b \in X$ and $\widetilde{ab}$ is the shortest arc joining $x$ and $y$ in the circle $S(x,h)\cap S(y,h)$. This implies that $d(x,a)=d(x,b)=d(y,a)=d(y,b)=h$ and therefore that $x,y$ are both vertices of the dual faces $\tau(a)\cap\tau(b)$. Since $\Phi$ is a standard ball polyhedron, then $\tau(a) = S(a,h)\cap\Phi$ and $\tau(b)=S(b,h)\cap\Phi$ intersect on the edge $\widetilde{xy}$. This proves that if $\tau(x)\cap \tau (y)\not= \emptyset$ then $\{x,y\}$ is an edge of $G_\Phi$. The proof of the converse is completely analogous.

Furthermore, if the vertex $x$ belongs to the dual face $\tau(y)=S(y,h)\cap\Phi$, then $d(x,y)=h$ and therefore the vertex $y$ belongs to the dual face $\tau(x)=S(x,h)\cap\Phi$.
\end{proof}
 
An important property of the embedding of the graph $G_\Phi$ in $R^3$ is that for every pair of points $x,y \in X$,
\begin{equation}\label{eq:embedding}
d(x,y)\leq h \text{ and }
d(x,y)= h \text{ iff } x \text{ is in the dual face of } y.
\end{equation}

A $3$-connected planar graph $G$ that admits an automorphism $\tau$ which is an involution ($x\notin \tau(x)$ and $x\in \tau(y)$ iff $y\in \tau(x)$) will be called an \emph{involutive self-dual graph}. 
A {\it metric embedding} of an involutive self-dual $G$ in $R^3$, is an embedding of the vertices $X$ of $G$ into $R^3$ as a geometric graph in such a way that \eqref{eq:embedding} holds. Examples of metric embeddings of self-dual graphs are $K_4$ with the vertices of the equilateral tetrahedron and the two concrete examples in Figure \ref{fig:reuleaux}.

The problem of characterizing when an involutive self-dual graph admits a metric embedding is an interesting one. The spherically self-dual polyhedra constructed by Lovasz in \cite{Lo1}, in connection with the distance problem, have this property.

\begin{theorem}\label{teo}
Let $X\subset R^3$ be the vertices of a metric embedding of the involutive self-dual graph $G$, then $\cap_{x\in X}B(x,h)$ is a Reuleaux polyhedron. Furthermore, the face polyhedral structure of $\cap_{x\in X}B(x,h)$ is lattice isomorphic to the polyhedral structure (points, edges and faces) of the planar graph $G$.
\end{theorem}

\begin{proof}
Let $\Phi=\cap_{x\in X}B(x,h)$. It will be enough to prove that the vertices of the boundary of $\Phi$ coincide with X. In this is case both lattices are isomorphic because, for every point in $X$, the dual face (as a point-set) of the face structure of the boundary of $\Phi$ and the dual face of the abstract polyhedron determined by the planar graph $G$ coincide.

For that purpose, let us prove first that the set $X\subset R^3$ admits $2n-2$ diameters, where $\abs{V} = n$. We start proving that for every $x\in X$, deg$(x, G)=$ deg$(x, D(X))$, where $D(X)$ is the graph whose vertex set is the set $X$ and whose edges are pairs $\{x,y\}\subset X$ such that the segment $xy$ is a diameter of $X$. Indeed, the degree of $x$ is equal to the number of faces of $G$ containing $x$ and since G is an involutive self-dual graph this number is equal to the number of vertices of the dual face of $x$, which is the degree of $x$ in $D(V)$. This proves that the number of edges of $G$ and of $D(V)$ coincide and, since $G$ is a self-dual graph with $2n-2$ edges, the Euler formula gives that $D(V)$ has $2n-2$ edges.

By the Gr\"unbaum-Heppes-Straszewicz Theorem (see \cite{KMP}), since every vertex of $X$ has degree grater than two, then the vertices of the boundary of $\Phi$ coincide with X.
\end{proof}

In Section \ref{sec:construct}, we shall construct an infinite family of Reuleaux polyhedra in $R^3$.

\begin{figure}[ht]
\includegraphics[width=.2\linewidth]{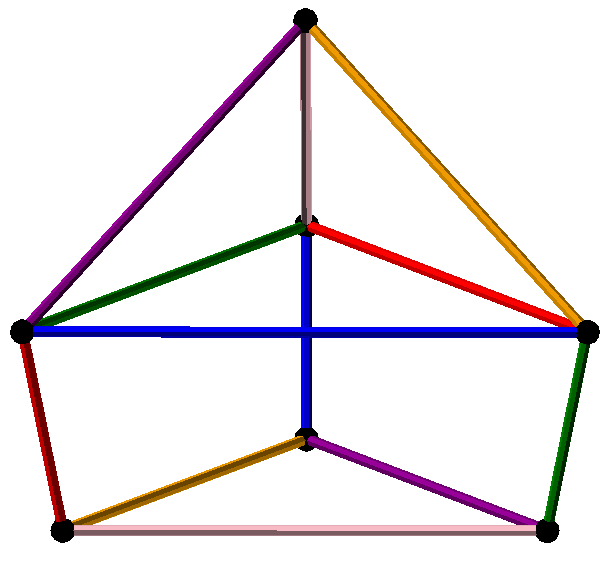}
\includegraphics[width=.2\linewidth]{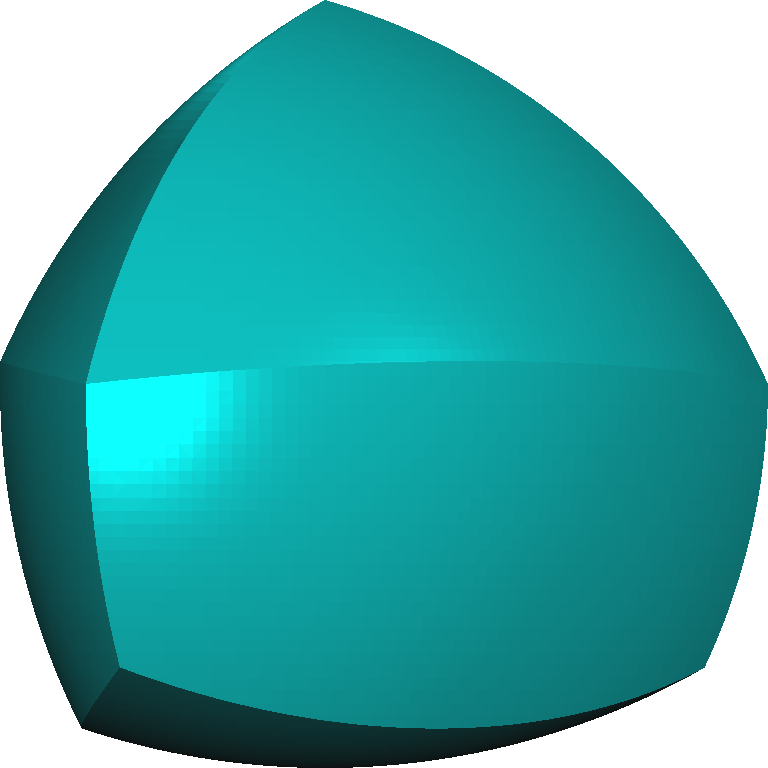}\qquad\quad
\includegraphics[width=.2\linewidth]{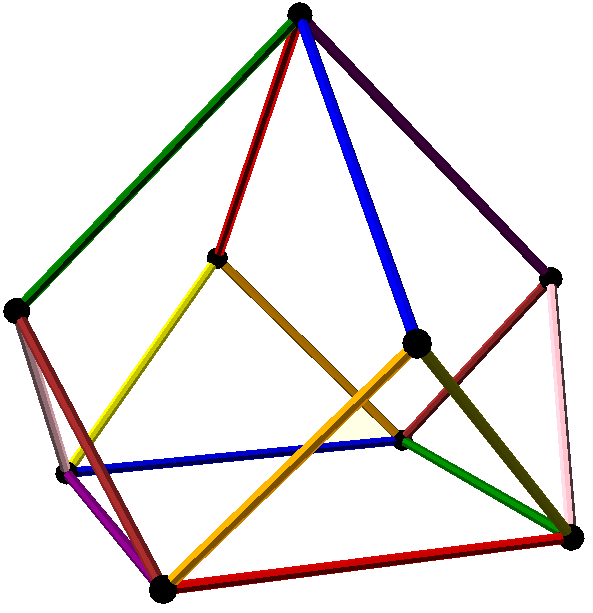}
\includegraphics[width=.2\linewidth]{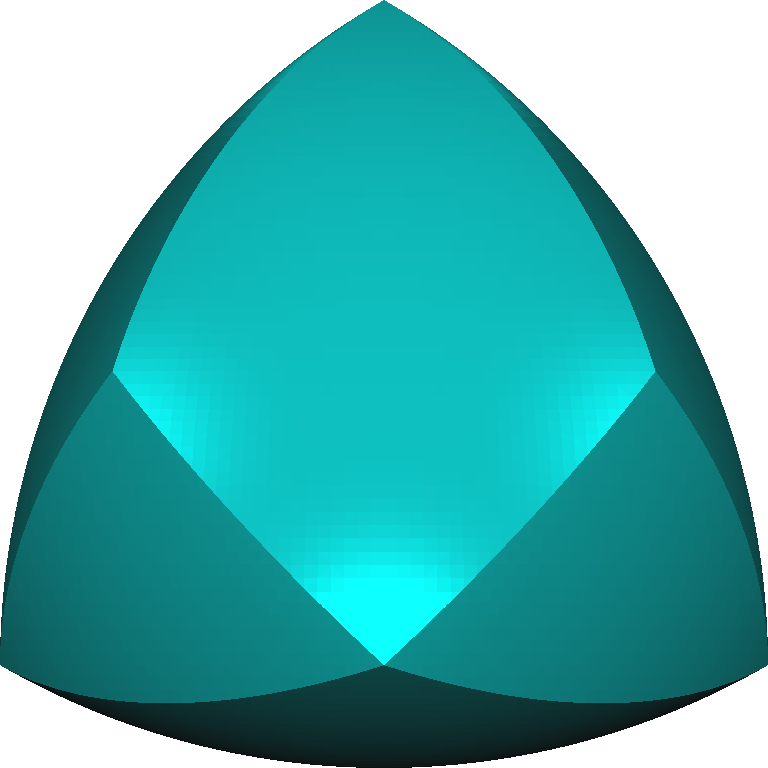}
\\[12pt]
\includegraphics[width=.2\linewidth]{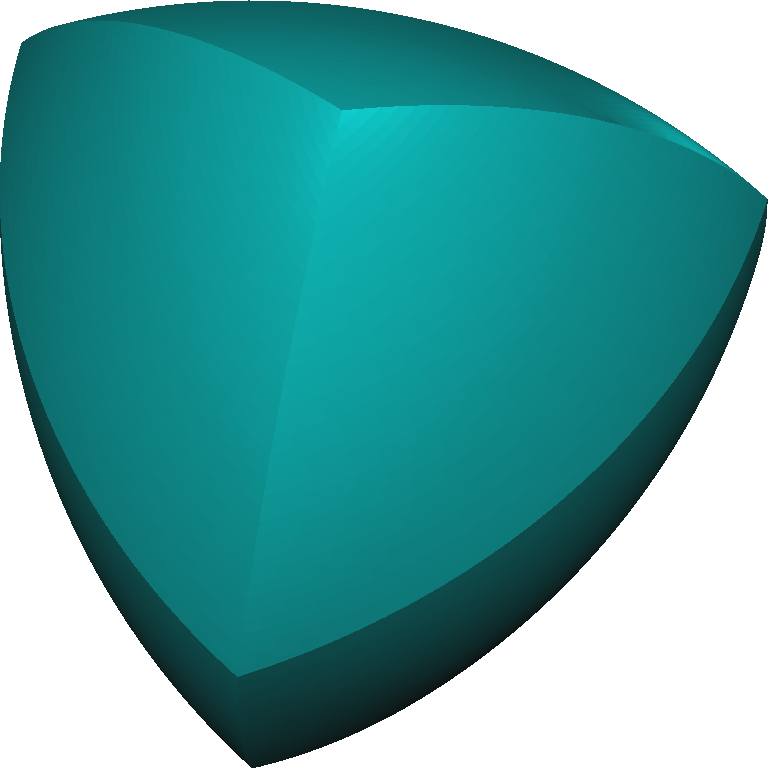}
\includegraphics[width=.2\linewidth]{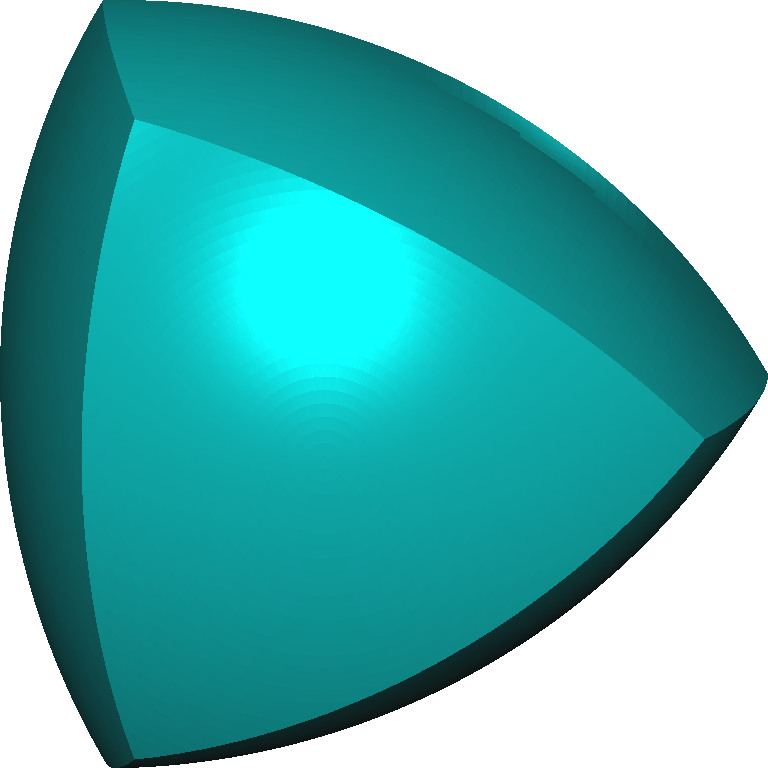}\qquad\quad
\includegraphics[width=.2\linewidth]{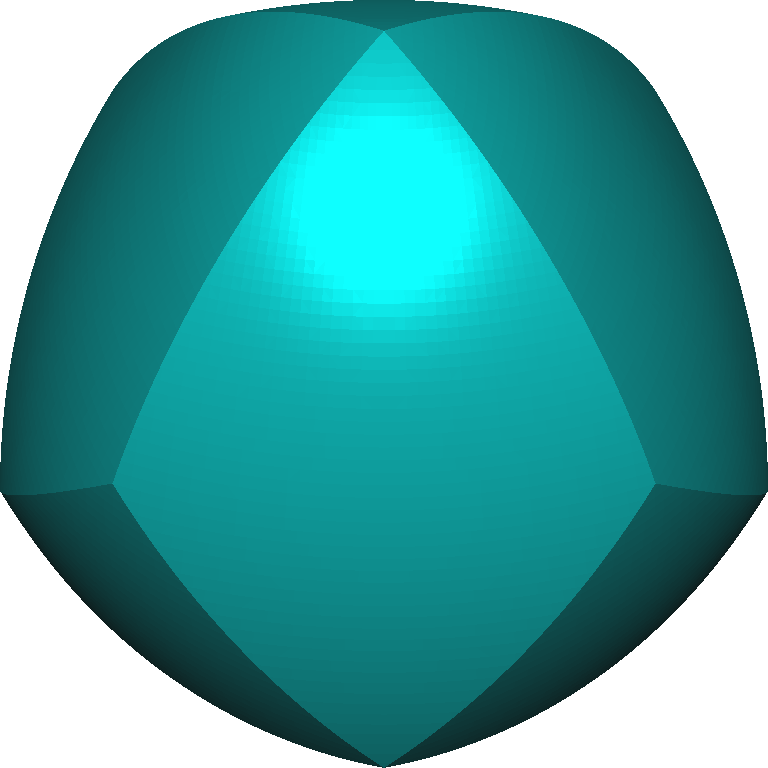}
\includegraphics[width=.2\linewidth]{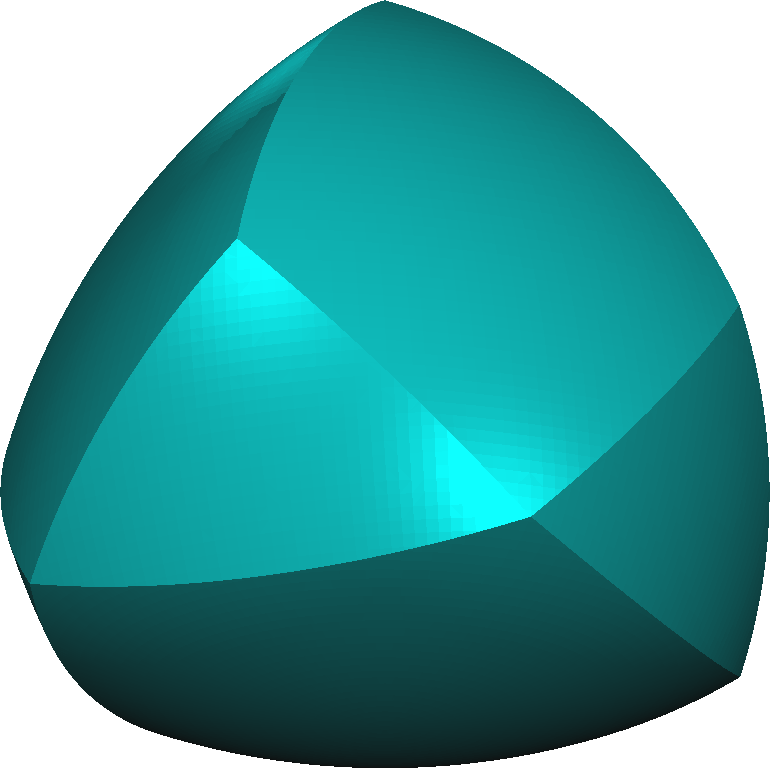}
\caption{Reuleaux polyhedra.}
\label{fig:reuleaux}
\end{figure}

\section{Meissner polyhedra}\label{sec:meipoly}

The two classic Meissner solids were constructed by performing surgery on one of the edges of each pair of dual edges of the Reuleaux tetrahedron. This procedure was described in Boltianski and Yaglom's book \cite{BY}. The purpose of this section is to generalize this procedure for Reuleaux polyhedra in $R^3$.

Let $G_\Phi\subset\bd\Phi$ be the metric embedding of the self-dual graph $G_\Phi$ as the singular points of the Reuleaux polyhedron $\Phi=\bigcap_{x\in X} B(x,h)$, where $X$ is the set of vertices of $V(G_\Phi)$.
Let us fix our attention on an edge $\widetilde{xy}\in E(G_\Phi)$. Then, there is dual edge $\widetilde{ab}\in E(G_\Phi)$ with the following properties:
\begin{itemize}
\item $d(x,a)=d(x,b)=d(y,a)=d(y,b)=h$.
\item The edge $\widetilde{xy}$ is contained in $S(a,h)\cap S(b,h)$, that is, $\widetilde{xy}$ is the sub-arc of the circle with center at $\frac{a+b}{2}$, between $x$ and $y$, contained in the plane orthogonal to the segment $\overline{ab}$.
\item Similarly, the edge $\widetilde{ab}$ is contained in $S(x,h)\cap S(y,h)$, that is, $\widetilde{ab}$ is the sub-arc of the circle with center at $\frac{x+y}{2}$, between $a$ and $b$, contained in the plane orthogonal to $\overline{xy}$.
\end{itemize}

Denote by $\tau(x)$ the dual face of the vertex $x$ in $\Phi$, that is 
$$\tau(x) = S(x,h)\cap \Phi.$$
By Lemma 1.1 (2) of \cite{BN}, $\tau(x)$ is a spherically convex closed subset of the sphere $S(x,h)$. As a subset of the boundary of $\Phi$, the face $\tau(x)$, of the ball polyhedron $\Phi$ is bounded by a finite number for sub-arcs of circles, each one an edge of $G_\Phi$. One of these edges is $\widetilde{ab}$.

Before we continue, we need a pair of definitions. If $\Phi$ is a convex body and $P\in\bd\Phi$, then a \emph{normal chord} of $\Phi$ at $P$ is a segment $PQ$ with $Q\in\bd\Phi\setminus\{P\}$ that is orthogonal to a supporting plane of $\Phi$ at $P$. Note that if $P$ is a regular point then this plane, and therefore $Q$, are unique. If $PQ$ is also a a normal chord of $\Phi$ at $Q$ then we call it a \emph{binormal chord}.

We now follow closely the procedure described in \cite{BY}, when they perform surgery on one of the edges of the Reuleaux tetrahedron.

Let $\Sigma_a\subset S(a,h)$ be the shortest geodesic joining $x$ and $y$. By the convexity of the faces of $\Phi$, $\Sigma_a\subset \tau (a) $. Similarly, let $\Sigma_b\subset \tau(b)$ be the shortest geodesic joining $x$ and $y$ in $S(b,h)$.
Denote by $\Sigma(xy)$ the region of the boundary of $\Phi$ between the arcs $\Sigma_a$ and $\Sigma_b$. Note that the edge $\widetilde{xy}$ is contained in $\Sigma(xy)$ and with the exception of these points, all points of the boundary of $\Phi$ contained in $\Sigma(xy)$ are regular and belong either to the sphere $S(a,h)$ or the sphere $S(b,h)$.

The above implies, in particular that if $P$ is a point in the interior of an edge $\widetilde{ab}\in E(G_\Phi)$ dual to $\widetilde{xy}\in E(G_\Phi)$ and if $PQ$ is a normal chord of $\Phi$ at $P$, then either $Q$ is a vertex of $\Phi$ and the length of $PQ$ is $h$ or, $Q$ belongs to $\Sigma(xy)$.

We denote by $W(x,y)$, the \emph{wedge along the segment} $\overline{xy}$, the surface of revolution with axis the line in $R^3$ containing the segment $\overline{xy}$, between the circular arc $\Sigma_a$ and the circular arc $\Sigma_b$. 
So the wedge $W(x,y)$ is the union of all circular arcs of radius $h$ between $x$ and $y$ with centers at points of the circular arc $\widetilde{ab}$ contained in the boundary of $\Phi$. We leave as an exercise to the reader to prove that the wedge $W(x,y)$ is contained in $\Phi$.
 
Let us modify the boundary of $\Phi$ by replacing $\Sigma(xy)$ with the wedge $W(x,y)$.
We denote by $\Sigma(\Phi)$ the surface obtained from the boundary of $\Phi$ by performing surgery on one edge of each pair of dual edges of the self-dual graph $G_\Phi$.

\begin{theorem}
Let $\Phi\subset R^3$ be a Reuleaux polyhedron. The surface $\Sigma(\Phi)$ obtained from the boundary of $\Phi$ by performing surgery on one edge of each pair of dual edges of the self-dual graph $G_\Phi$ is the boundary of a constant width body.
\end{theorem}

\begin{proof} The proof consist of two steps. In the first step we prove that for every point $P$ in the surface $\Sigma(\Phi)$, there is a point $Q$ in the surface $\Sigma(\Phi)$ such that $d(P,Q)=h$. In the second step we prove that 
the diameter of the surface $\Sigma(\Phi)$ is equal to $h$. If this is so, then by Pál's Theorem, there is a body of constant width containing the surface $\Sigma(\Phi)$ which of course has $\Sigma(\Phi)$ as its boundary.

For the first part of the proof, note that a point $P$ in the surface $\Sigma(\Phi)$ belongs either to a face $\tau(x)$, for some vertex $x$ of $G_\Phi$ or belongs to a wedge $W(x,y)$ for some edge $\widetilde{xy}$ of $G_\Phi$. In the first case case, $d(P,x)=h$ and in the second case, $d(P,Q)=h$ for some point $Q$ in the dual edge $\widetilde{ab}$ of $\widetilde{xy}$.

For the second part of the proof, suppose $PQ$ is a diameter of $\Sigma(\Phi)$. Then the planes orthogonal to $PQ$ at either $P$ or $Q$ are support planes of $\Sigma(\Phi)$. Furthermore, since the surface $\Sigma(\Phi)$ is contained in $\Phi$, by the strict convexity of $\Phi$, all the points strictly between $P$ and $Q$ are interior points of $\Phi$.

At this point, we need a classification of the points of $\Sigma(\Phi)$. First, we have the vertex points, second we have the points of $\Sigma(\Phi)$ which are in the relative interior of some edge $\widetilde{ab}\in E(G_\Phi)$. All other points of $\Sigma(\Phi)$ are regular points. From this last class, we have those which are in the interior of some face $\tau(x)$ of $\Phi$ and therefore belong to $S(x,h)$ for some vertex $x$ of $\Phi$, those which are in the interior of a wedge $W(x,y)$ and, finally, those regular points of $\Sigma(\Phi)$ which are in the intersection of a wedge $W(x,y)$ and a face $\tau(a)$. 

Suppose first that $P$ is a regular point of $\Sigma(\Phi)$. If $P$ belongs to a face of $\Phi$, then the chord $PQ$ is a normal chord of $\Phi$ at $P$ and therefore there is a vertex $X$ of $\Phi$ between $P$ and $Q$. Since $X\in\bd\Phi$ and since every point strictly between $P$ and $Q$ is an interior point of $\Phi$, then $X=Q$ and therefore the length of $PQ$ is $h$. Similarly, if $P$ is a point of the wedge $W(x,y)$, then the chord $PQ$ is normal to the surface $W(x,y)$ at $P$ and therefore there is a point $X$ in the edge $\widetilde{ab}\in E(G_\Phi)$ dual to $\widetilde{xy}\in E(G_\Phi)$ which is between $P$ and $Q$. As before, since $X\in\bd\Phi$, then $X=Q$ and therefore the length of $PQ$ is $h$.

Suppose now neither $P$ nor $Q$ are regular points of $\Sigma(\Phi)$. If both $P$ and $Q$ are vertices of $\Phi$ then the length of $PQ$ is smaller than or equal to $h$, so we may assume that $P$ is in the interior of the edge $\widetilde{ab}$. If this is so, since $PQ$ is normal to the surface $\Sigma(\Phi)$ at $P$, then it is also normal to $\Phi$ at $P$. Since $Q$ is either a vertex of $\Phi$ or lies in some edge of $\Phi$, we have that $PQ$ is a chord of $\Phi$. In fact, it is a normal chord $\Phi$, hence either $Q$ is a vertex and then the length of $PQ$ is $h$ or $Q$ belongs to $\Sigma(xy)$, which is impossible by construction. This completes the proof.
\end{proof}

We define a Meissner polyhedron as any constant width body that can be obtained from a Reuleaux polyhedron $\Phi$ in $R^3$ by performing surgery on one edge of each pair of dual edges of the self-dual graph $G_\Phi$.

\begin{figure}[ht]
\includegraphics[width=.3\linewidth]{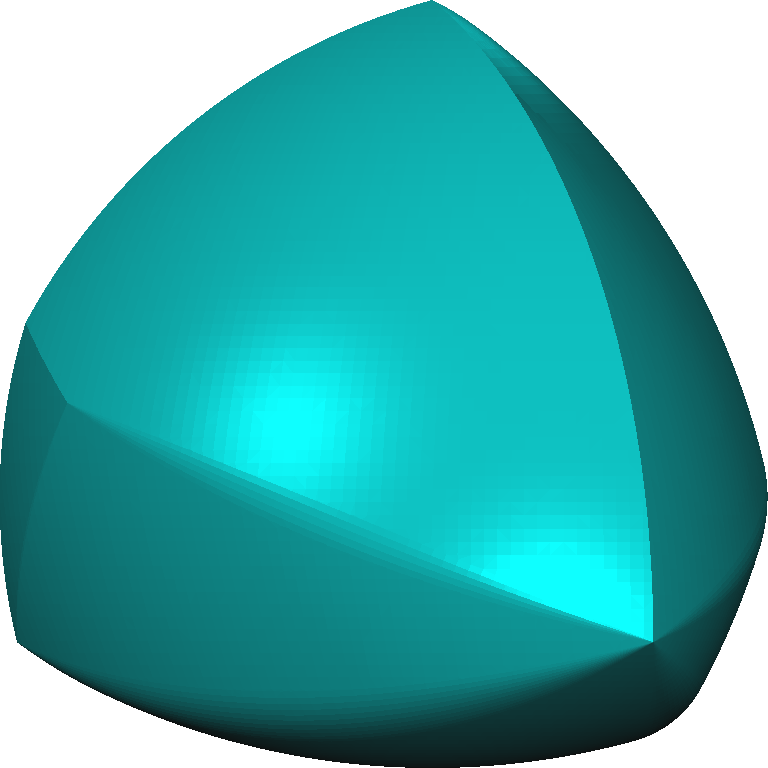}
\includegraphics[width=.3\linewidth]{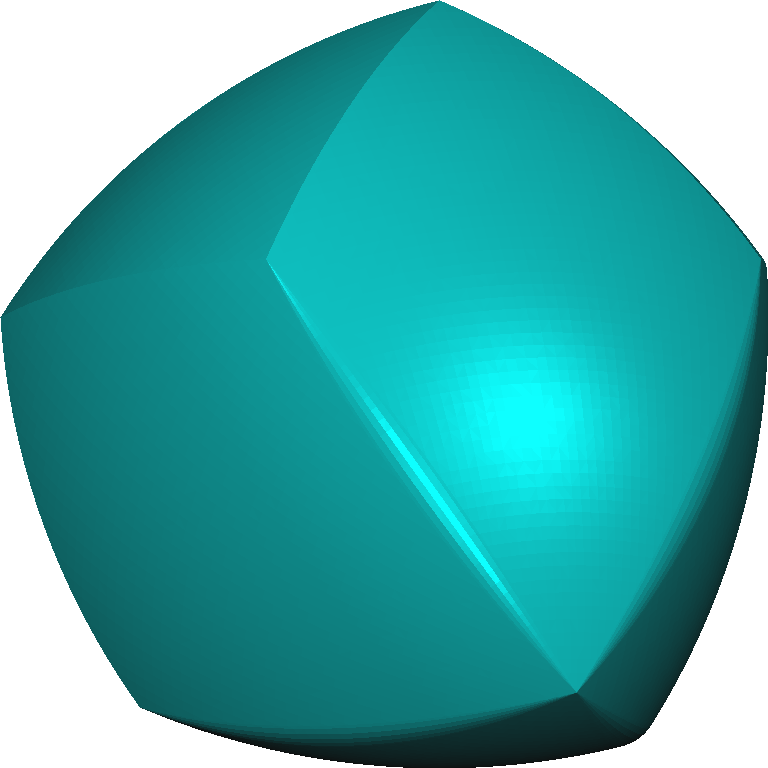}
\includegraphics[width=.3\linewidth]{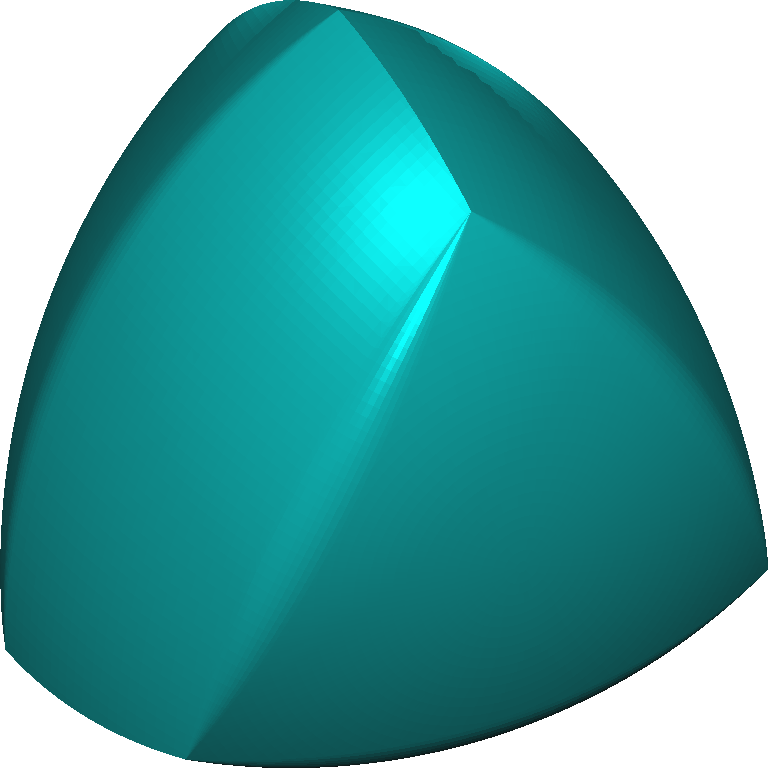}
\caption{Meissner polyhedra.}
\label{fig:meissner}
\end{figure}

In \cite{S}, Sallee showed that every $3$-dimensional smooth convex body of constant width $h$ can be closely approximated by Reuleaux polyhedron (with arbitrarily small edges). Therefore every $3$-dimensional constant width body can be closely approximated by a Meissner polyhedron.

A body $\Psi\subset R^3$ of constant width is called a \emph{Meissner solid} if it has the property that the smooth components of its boundary have their smaller principal curvature constant. Here smooth means twice continuously differentiable. Clearly, every Meissner polyhedron is a Meissner solid.

The Blaschke-Lebesgue problem consists of minimizing the volume in the class of convex bodies of fixed constant width. Anciaux and Guilfoyle \cite{AG} proved that a minimizer of the Blaschke-Lebesgue problem is always a Meissner solid.
On the other hand, Shiohama and Takagi \cite{ST} proved that a non-spherical surface with one constant principal curvature must be a canal surface, that is, the envelope of a one-parameter family of spheres or equivalently, a tube over a curve (i.e. the set of points which lie at a fixed distance from this curve). They are made of spherical caps of radius $h$ and surfaces of revolution over a circle of radius $h$, exactly like the Meissner polyhedra.

\section{Constructing constant width bodies from Reuleax polygons}\label{sec:construct}

Let $P$ be a Reuleaux polygon of width $1$ with vertices $p_1,p_2,\dots,p_n$ as in Figure \ref{fig:pentagon}. We will assume that $P\subset R^2\times\{0\}\subset R^3$. Of course, $n$ is an odd integer and $P=\bigcap_{i=1}^n B(p_i,1)\cap\{z=0\}$.

\begin{figure}[ht]
\includegraphics[width=0.3\textwidth]{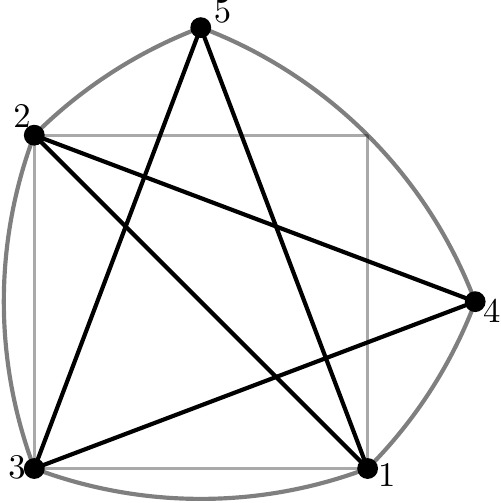}
\caption{A Reuleaux polygon $P$ with five vertices.}
\label{fig:pentagon}
\end{figure}

Let us consider the \emph{farthest-point Voronoi diagram of $P$} (see e.g. \cite{Bro}). In other words: for each vertex $p_i$ of $P$, consider the set of all points $x\in P$ with the property that
$$d(x,p_i)\geq \max\{d(x,p_j)\mid 1\leq j \leq n\}.$$
This gives a cell decomposition of $P$ into $n$ convex cells, each one of them containing the corresponding circle arc $S(p_i,1)\cap P$.
The boundary between any two of these convex cells is a straight line edge, and the collection of these edges gives rise to the embedding of a tree with straight line edges and vertices denoted by $V(P)$. See Figure \ref{fig:vorodela}(a).

\begin{figure}[ht]
\includegraphics[width=0.3\textwidth]{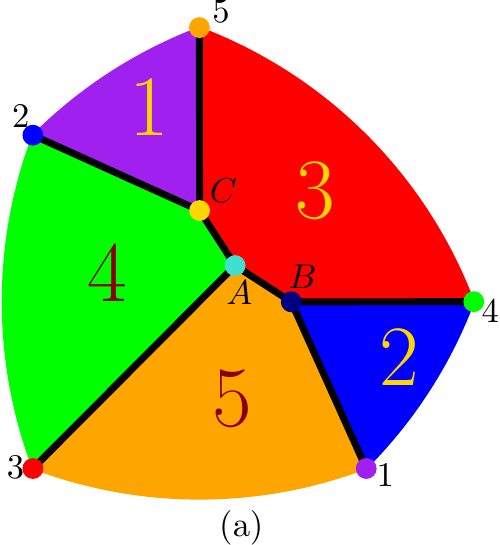}\qquad
\includegraphics[width=0.3\textwidth]{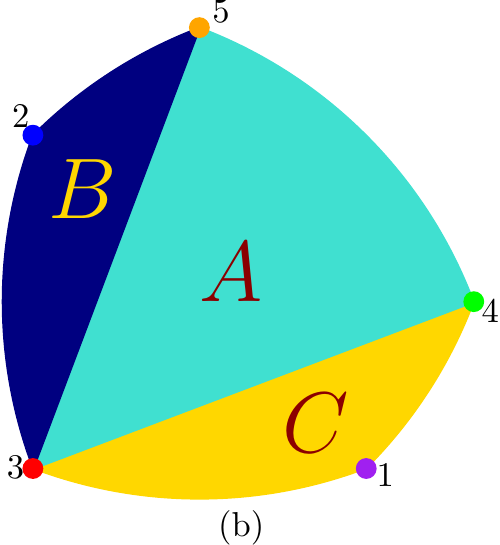}
\caption{Voronoi diagram and Delaunay triangulation.}
\label{fig:vorodela}
\end{figure}

The \emph{farthest-point Delaunay triangulation of $P$} is the planar dual of the farthest-point Voronoi diagram of $P$ (see e.g. \cite{Epp}). It gives rise to a family $F$ of subsets of $\{p_1,\dots,p_n\}$ whose convex hulls divide $\conv\{p_1,\dots,p_n\}$. See Figure \ref{fig:vorodela}(b). It turns out that $T\in F$ if and only if $\abs{T}\ge 3$ and there is a disk $D_T$ containing $P$ whose boundary intersects $P$ exactly at $T$.
Furthermore, if $c_T$ is the center of $D_T$ and $r_T$ is the radius of $D_T$, then the collection $\{c_T \mid T\in F\}$ is precisely $V(P)\setminus\{p_1,\dots,p_n\}$.

Now consider the $3$-dimensional ball polyhedron
$$\bar P = \bigcap_{i=1}^n B(p_i,1).$$
As in Section \ref{sec:ballpoly}, let $V(\bar P)$ be the set of $0$-singular points of the boundary of $\bar P$.
It is not difficult to verify that $V(\bar P)=\{(c_T, \pm\sqrt{1-r_T^2}) \mid T\in F\}$ and that the orthogonal projection of the cell decomposition of the boundary of $\bar P$ coincides precisely with the cell decomposition of the Voronoi diagram for the farthest point discussed above.

Define 
$$P^+=\bar P\cap \{z\geq 0\}.$$
and for every $1\leq i\leq n$, let $\sigma_i=S(p_i,1) \cap P^+$ be the spherical face of the boundary of $P^+$.

\begin{lemma}
The convex body $P^+$ has the following properties:
\begin{itemize}
\item the diameter of $P^+$ is $1$,
\item for every point $x \in\bd P^+ \cap \{z>0\},$ there is a point $y$ in the boundary of the Reuleaux polygon $P$ such that $d(x,y)=1$.
\end{itemize}
\end{lemma}

\begin{proof}
The second statement is obvious. For the proof of the first statement, let $xy$ be a diameter of $P^+$. Then we may assume without loss of generality that $x$ belongs to the boundary of the polygon $P$ and $y \in\bd P^+ \cap \{z>0\}$. If $y$ is a regular point of the boundary of $P^+$, then $d(x,y)=1$. If $y$ is a singular point of $P^+$, since $x$ is in the boundary of $P$, then $x$ is a vertex of $P$ and hence $d(x,y)=1$.
\end{proof}

So, by Pál's theorem, there is a unique body $\Psi$ of constant width $1$ containing $P^+$. Furthermore, by the above, $\bd P^+ \cap \{z>0\}$ is contained in the boundary of 
$\Psi$, so $\Psi \cap \{z\geq 0\}=P^+$. Now we want to describe the bottom of $\Psi$, namely $\Psi\cap \{z\leq 0\}$.
It turns out that the bottom of $\Psi$ is determined by the binormal chords at singular points of the boundary of $P^+$.
That is, for every point $x\in$ $\bd\Psi \cap \{z\leq 0\}$ there is a singular point $y\in\bd P^+ \cap \{z>0\}$, such that $d(x,y)=1$.

Recall that $F$ is the farthest-point Delaunay triangulation of $P$ and let $T\in F$. The set of normal lines of $P^+$ at the vertex $v_T=(c_T,\sqrt{1-r_T^2})$ is equal to the cone $\bigtriangleup_T$ with apex at the point $v_T$ though $\conv T$, therefore $\sigma_T=\bigtriangleup_T \cap S(v_T,1)\subset\bd\Psi$ is a spherical cap face of $\Psi$.

Assume that $v,v'\in V(P)$ are adjacent in the tree defined by the farthest-point Voronoi diagram of $P$. Let $x$ be a point on the edge $\widetilde{vv'}$, we wish to determine the set of normal lines of $P^+$ at $x$. There are two possibilities.

If $v=c_T$ and $v=c_{T'}$ for some $T,T'\in F$, then $T$ and $T'$ share a side of the form $p_ip_j$. Thus, the set of normal lines of $P^+$ at $x$ is equal to the cone with apex $x$ through the segment $\overline{p_ip_j}$.
This implies that, in the boundary of $\Psi$, the face $\sigma_T$ is connected with the face $\sigma_{T'}$ through a surface of revolution with axis the line through the segment $\overline{p_ip_j}$.

\begin{figure}[ht]
\includegraphics[width=0.3\textwidth]{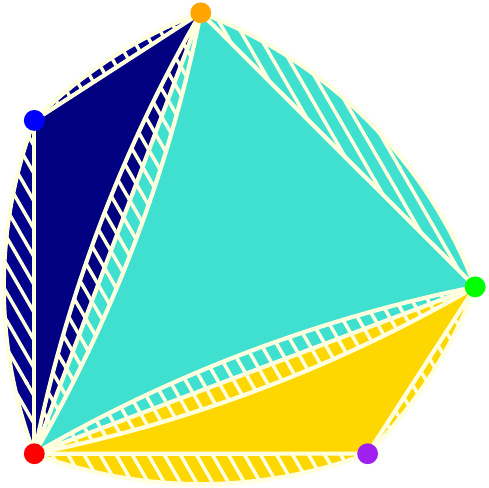} 
\caption{$\Psi$ from the bottom.}
\label{fig:bottombull}
\end{figure}

If $v=c_T$ for some $T\in F$ and $v'=p_i$ for some $i$, we may reason similarly to conclude that, in the boundary of $\Psi$, the face $\sigma_T$ is connected with the face $\sigma_i$ through a surface of revolution with axis the line through the segment of the form $\overline{p_jp_k}$ opposite to $p_i$. These surfaces of revolution are illustrated in Figure \ref{fig:bottombull}.

So the boundary of $\Psi$ consists of the spherical caps $\sigma_1,\dots,\sigma_n$ and $\sigma_T$ with $T\in F$ and surfaces of revolution on the bottom part of $\Psi$. A picture of such a body is shown in \ref{fig:bull}.

\begin{figure}[ht]
\includegraphics[width=0.3\textwidth]{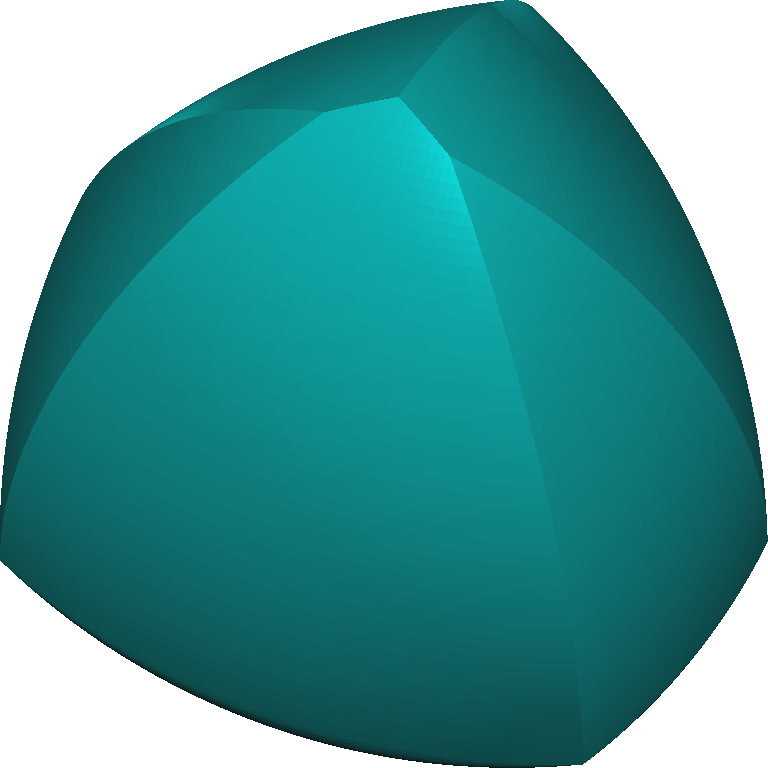}\qquad
\includegraphics[width=0.3\textwidth]{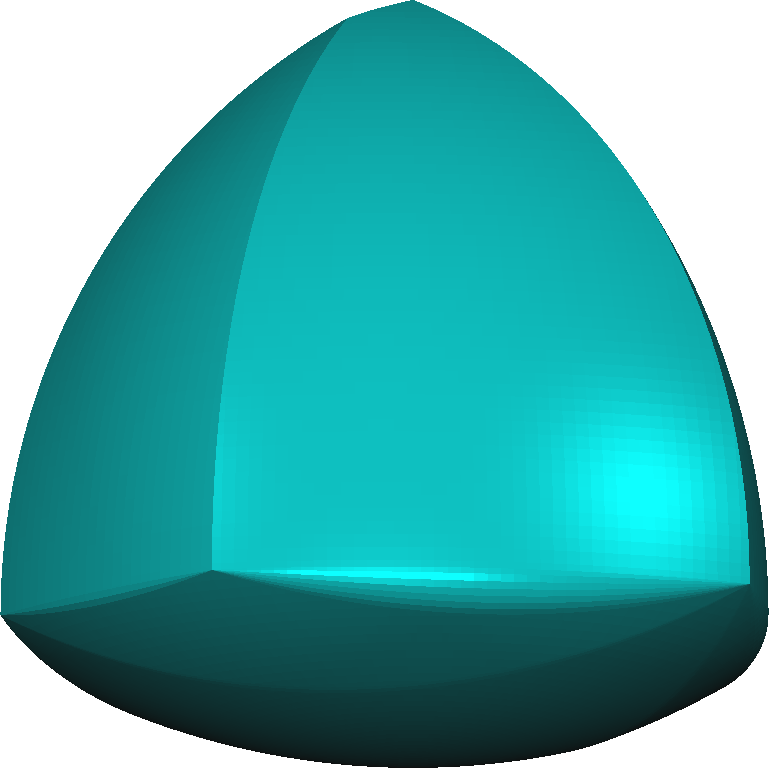} 
\caption{A constant width body obtained from a Reuleaux pentagon.}
\label{fig:bull}
\end{figure}

It is easy to see, from the above discussion, that the set of vertices of $\Psi$ is $V=\{p_i\mid 1\le i\le n\}\cup\{v_s\mid T\in F\}$. So, $$\Phi=\bigcap_{v\in V}B(v,1)$$ is a Reuleaux polyhedron and therefore induces a self-dual graph $G_\Phi$ (see Figure \ref{fig:bullgraph}).
Actually, $V\subset R^3$ is a metric embedding of $G_\Phi$ and of course, $\Psi$ is a Meissner polyhedron because it can be obtained from the Reuleaux polyhedron $\Phi$ by performing surgery on one edge of each pair of dual edges of $\Phi$, as shown in Section \ref{sec:meipoly}.

\begin{figure}[ht]
\includegraphics[width=0.3\textwidth]{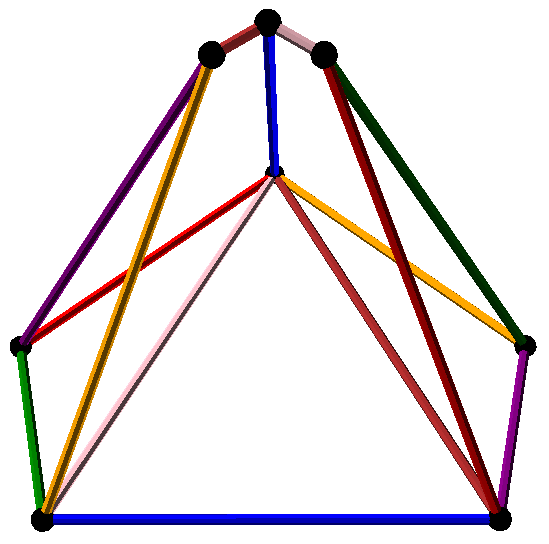} 
\caption{Metric embedding of the self-dual graph $G_\Phi$.}
\label{fig:bullgraph}
\end{figure}

As a second example we show the Voronoi diagram and Delaunay triangulation of a Reuleaux polygon with $7$ vertices in Figure \ref{fig:deer}. We can also interpret these two figures as the top and bottom view of the corresponding Meissner solid.

\begin{figure}[ht]
\includegraphics[width=0.3\textwidth]{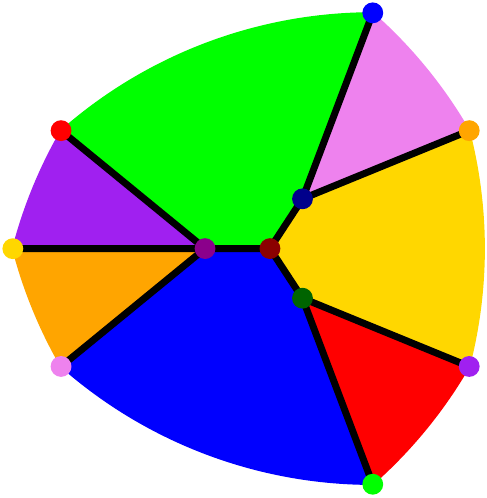}
\qquad
\includegraphics[width=0.3\textwidth,height=0.3\textwidth]{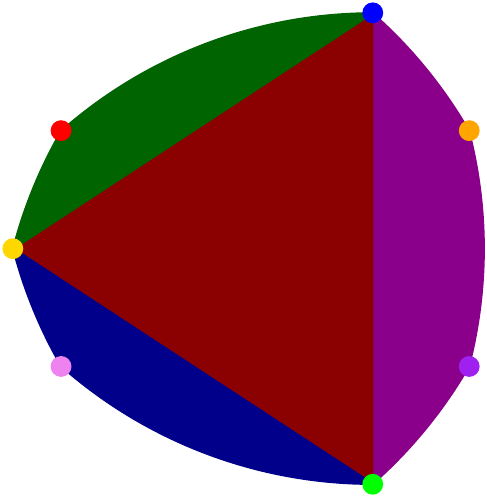}
\\[12pt]
\includegraphics[width=0.3\textwidth]{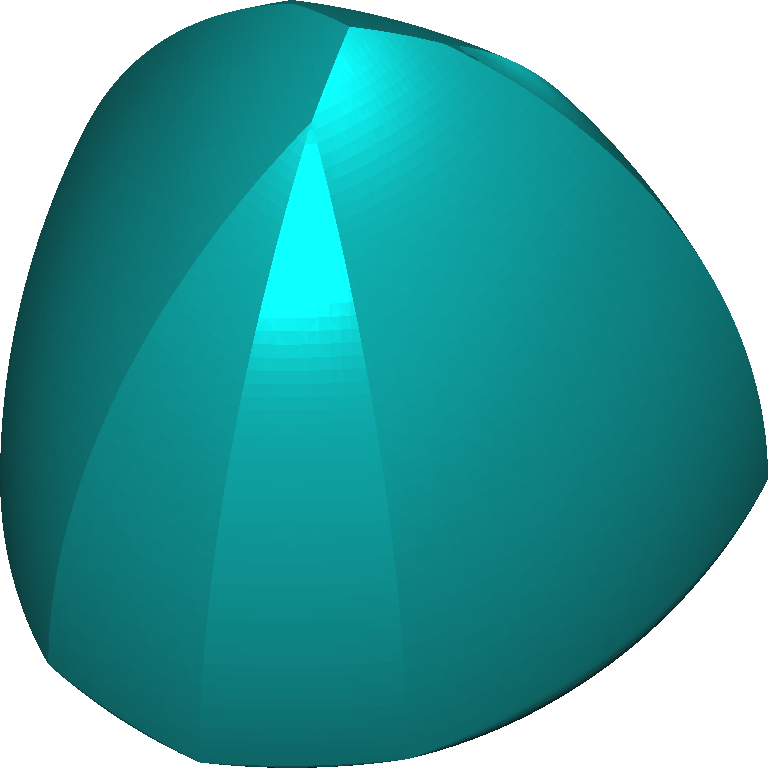}\qquad
\includegraphics[width=0.3\textwidth]{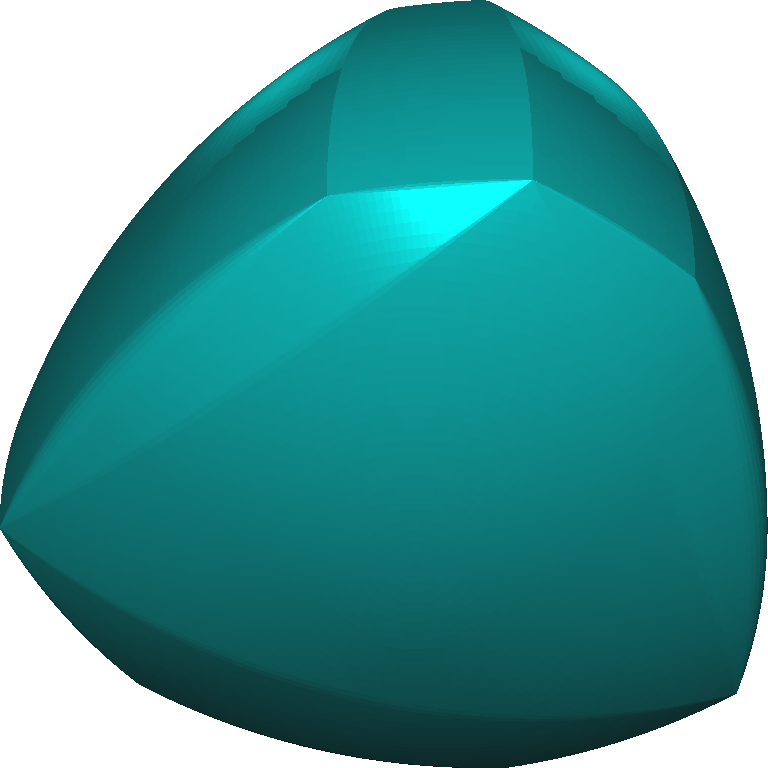}
\caption{Voronoi diagram and Delaunay triangulation of a Reuleaux heptagon and the body of constant width obtained from it.}
\label{fig:deer}
\end{figure}

Summarizing, we have proved the following.

\begin{theorem}\label{sum}
Given a Reuleaux polygon $P\subset R^2$, there is a finite procedure (using the Voronoi Diagram and the Delaunay triangulation) to construct a Meissner polyhedra $\tilde P\subset R^3$ such that
$\tilde P\cap R^2=P$.
\end{theorem}

\section{Acknowledgements}

This research was supported by CONACYT project 166306 and PAPITT-UNAM project IN112614.

\end{document}